\documentclass[12pt]{tran-l}
\usepackage{amsthm,amsmath,amssymb,graphicx,amscd,mathrsfs}
\newtheorem{theorem}{Theorem}[section]%
\newtheorem{corollary}[theorem]{Corollary}%
\newtheorem{lemma}[theorem]{Lemma}%
\theoremstyle{definition}%
\theoremstyle{remark}%
\newtheorem{remark}[theorem]{Remark}%
\numberwithin{equation}{section}%
\renewcommand{\r}[1]{\mathbb{R}^{#1}}
\newcommand{\ccr}[1]{C_{c}^{\infty}\left(\r{#1}\right)}
\newcommand{\sr}[1]{\mathscr{S}\left(\r{#1}\right)}
\newcommand{\sh}[1]{\mathscr{S}\left(\h{#1}\right)}

\newcommand{\sprimer}[1]{\mathscr{S}'\left(\r{#1}\right)}
\newcommand{\h}[1]{\mathbb{R}_{+}^{#1}}
\newcommand{\cch}[1]{C_{c}^{\infty}\left(\h{#1}\right)}
\newcommand{\bigo}{\operatorname{O}}
\newcommand{\id}{\operatorname{I}}
\newcommand{\pv}{(\operatorname{P.V.})}
\newcommand{\sgn}{\operatorname{sgn}}
\newcommand{\supp}{\operatorname{supp}}
\newcommand{\vol}{\operatorname{vol}}
\newcommand{\hilbert}{\mathcal{H}}
\newcommand{\norton}{\mathcal{P}}
\newcommand{\radon}{\mathcal{R}}
\newcommand{\sonar}{\mathcal{S}}
\newcommand{\n}{\mathcal{N}}
\newcommand{\m}{\mathcal{M}}
\renewcommand{\k}{\mathcal{K}}%
\renewcommand{\l}{\mathcal{L}}
\newcommand{\q}{\mathcal{Q}_{p}}

\title[Inversion of spherical means]{Inversion of the spherical mean transform with sources on a hyperplane}
\author{Aleksei Beltukov}
\address{Department of Mathematics, University of the Pacific, Stockton CA 95211}
\email{abeltuko@pacific.edu}
\date{\today}
\begin{document}
\maketitle
\begin{abstract}
The object of this study is an integral operator~$\sonar$ which averages functions in the Euclidean upper half-space~$\h{n}$ over the half-spheres centered on the topological boundary~$\partial \h{n}$.  By generalizing Norton's approach to the inversion of arc means in the upper half-plane, we intertwine~$\sonar$ with a convolution operator~$\norton$.  The latter integrates functions in~$\r{n}$ over the translates of a paraboloid of revolution.  Our main result is a set of inversion formulas for~$\norton$ and~$\sonar$ derived using a combination of Fourier analysis and classical Radon theory.  These formulas appear to be new and are suitable for practical reconstructions.
\end{abstract}

\section{Introduction}
\label{s:intro}


The transform of the title~$\sonar$ averages functions in the Euclidean upper half-space \( \h{n+1} \equiv \r{n} \times \h{1} \) \( (n \geq 1) \) over the half-spheres centered on the topological boundary~$\partial \h{n+1}$.  Courant and Hilbert's well known result in~\cite{Courant-Hilbert-1962} implies that~$\sonar$ is injective on the space of continuous functions.  If one further restricts the domain of~$\sonar$ to rapidly decreasing functions, as done in~\cite{Andersson-1988}, then~$\sonar$ can be shown to have a continuous inverse~$\sonar^{-1}$.  The main purpose of this article is to present a novel formula for that inverse which compliments the ones in~\cite{Norton-1980a}, \cite{Fawcett-1985}, \cite{Andersson-1988}, \cite{Nessibi-et-al-1995}, \cite{Xu-Wang-2005}, and~\cite{Klein-2003} in several important ways:
\begin{enumerate}
    \item Our inversion of spherical means holds in all dimensions\footnote{The inversion formulas in~\cite{Norton-1980a} and~\cite{Xu-Wang-2005} are derived only in dimensions two and three, respectively.}, has explicit form involving the adjoint operator~$\sonar^{*}$ \footnote{The inversion formula in~\cite{Andersson-1988} is stated in Fourier domain.  In practice it is often preferable to have a back-projection type formula (that is, involving the adjoint) as it leads to more stable reconstruction algorithms.}, and applies without restrictions to all rapidly decreasing functions in the Schwartz space~$\sh{n+1}$.\footnote{In~\cite{Norton-1980a} the spherical means are assumed to be band-limited in one variable; as shown in~\cite{Klein-2003}, the inversion formulas in~\cite{Fawcett-1985}, \cite{Andersson-1988} and~\cite{Nessibi-et-al-1995} require strong analytic constraints to be imposed on the functional domains \emph{in addition} to rapid decrease.}
    \item The validity of our inversion formula follows directly from the classical Radon inversion.  We thus establish a new relation between the spherical mean transform~$\sonar$ and the classical Radon transform~$\radon$ which, together with~\cite{Denisjuk-1999a} and~\cite{Beltukov-Feldman-2009}, brings the count for such relations to three.
    \item Our formula for~$\sonar^{-1}$ may be expressed as a composition of a certain convolution operator~$\norton^{-1}$ with a pair of simple one-dimensional mappings. That compositional structure lends itself to a robust Fourier-based algorithm suitable for tomographic reconstructions.
\end{enumerate}


Although our interest in spherical means is purely mathematical, we would like to mention a few practical applications before delving into the theory.  The following example, taken from~\cite{Fawcett-1985}, explains why~$\sonar$ continues to draw attention almost fifty years after having been first considered by Courant and Hilbert; it also explains why we refer to the centers of the half-spheres as \emph{sources}.

Consider the following nonhomogeneous hyperbolic equation with a delta-source at
\(
  (x_{s},y_{s},0) \in \r{3}
\)
and zero initial conditions:
\begin{equation} \label{e:hyperbolic}
  \left\{
    \begin{split}
      &\frac{\partial^{2} u}{\partial t^{2}} - \nabla^{2} u + f(x,y,z) \, u =
       \delta(x - x_{s}) \, \delta(y - y_{s}) \, \delta(z) \, \delta(t), \\
      &u(x,y,z,0) = \frac{\partial u}{\partial t}(x,y,z,0) = 0.
    \end{split}
  \right.
\end{equation}
Think of~$(x_{s},y_{s},0)$ as a variable point in the $xy$-plane and denote by~$g(x_{s},y_{s},t)$ the solution of Problem~\eqref{e:hyperbolic} evaluated at the location of the source:
\(
  \{ x = x_{s}, y = y_{s}, z = 0 \}.
\)
The \emph{inverse scattering problem} associated with Equation~\eqref{e:hyperbolic} is to determine the coefficient~$f(x,y,z)$ from~$g(x_{s},y_{s},t)$.  This problem and its close variants arise in sonar and radar imaging (see~\cite{Louis-Quinto-2000} and~\cite{Cheney-2001}, respectively), ultrasound diagnostics~\cite{Norton-1980a}, and seismic tomography~\cite{Fawcett-1985}.  As explained in~\cite{Fawcett-1985}, by making the Born approximation one can interpret the data~$g(x_{s},y_{s},t)$ as the average of~$f(x,y,z)$ over the sphere of radius~$t/2$ centered at~$(x_{s},y_{s},0)$.  If $f$ is compactly supported in the upper (or lower) half-space, which is frequently the case \footnote{For instance, $f$ could describe the acoustic reflectivity of underwater or underground environment both of which can be thought of as being semi-infinite in extent.}, the averages over spheres become averages over half-spheres.  In that case the solution of the linearized inverse scattering problem for Equation~\eqref{e:hyperbolic} is equivalent to the inversion of the spherical mean transform~$\sonar$.


The inspiration for our results comes from Norton's work in~\cite{Norton-1980a} on inversion of~$\sonar$ in~$\h{2}$.  In Section~\ref{s:preliminaries} we will show how Norton's approach can be generalized to an arbitrary number of dimensions.  Specifically, we will define a pair of mappings which intertwine~$\sonar$ with an operator~$\norton$ acting on functions in the full Euclidean space.  The latter operator is of convolution type: it integrates functions in~$\r{n+1}$ over the translates of a paraboloid.  We will exploit the convolutional nature of~$\norton$ in Section~\ref{s:fourier} where we find~$\norton^{-1}$ using distribution theory.  Our main result, the formula for~$\sonar^{-1}$, follows from the formula for~$\norton^{-1}$ and has the general form
\[
  f = y \, ( \sonar^{*} \circ \k_{y} \circ \sonar) (f), \quad f(x,y) \in \sh{n+1},
\]
where~$\k_{y}$ is a one-dimensional operator defined in Section~\ref{s:results}. The advantage of our inversion formula over a similar result in~\cite{Nessibi-et-al-1995} is that our one-dimensional operator~$\k$ is a simple pull-back of the classical $\Lambda$-operator, whereas the operator~$K$ in the formula \( f = (\sonar^{*} \circ K \circ \sonar)(f) \) derived in~\cite{Nessibi-et-al-1995} is a much more complicated fractional power of the Laplacian.  Furthermore, as we show in Section~\ref{s:radon}, the presence of the $\Lambda$-operator, common in Radon theory, is not coincidental: our inversion formula for the convolution~$\norton$ follows directly from the classical Radon inversion formula. In fact, we show in Theorem~\ref{t:ip_radon} that the spherical mean transform~$\sonar$ is, in a sense, more naturally inverted by the classical Radon transform.
\par
The author is deeply indebted to Dr. Eric Todd Quinto of Tufts University for detailed and helpful suggestions concerning the preparation of the manuscript.  He would also like to thank his colleague Dr. Jialing Dai for stimulating mathematical discussions of Radon and spherical mean transforms.

\section{Preliminaries}
\label{s:preliminaries}

Henceforth it will be convenient to identify
\begin{math}
  \r{n+1} \equiv \r{n} \times \r{1}
\end{math}
and, similarly,
\begin{math}
  \h{n+1} \equiv \r{n} \times \h{1} =
  \{ (x,y) \mid x \in \r{n}, \: y > 0 \}.
\end{math}
The spaces of smooth compactly supported function in~$\r{n+1}$ and~$\h{n+1}$ will be denoted by~$\ccr{n+1}$ and~$\cch{n+1}$, respectively; for Schwartz spaces we will use the symbols~$\sr{n+1}$ and~$\sh{n+1}$.
\par
Let~$f \in \sh{n+1}$ $(n \geq 1)$ be the unknown function to be determined from its spherical means:
\(
  g = \sonar(f).
\)
We parameterize the spherical mean data as follows:
\begin{equation} \label{e:s}
  g(x,y) =
  \frac{2}{|S^{n}|} \, y^{1-n} \,
  \int\limits_{|u| < y}
    \frac{f \left( x + u, \sqrt{y^{2} - |u|^{2}} \right)}
         {\sqrt{y^{2} - |u|^{2}}} \,
  du.
\end{equation}
In Equation~\eqref{e:s}
\begin{math}
  |u| = \sqrt{u_{1}^{2} + \ldots + u_{n}^{2}}
\end{math}
is the Euclidean norm in~$\r{n}$,
\begin{math}
  du = du_{1} \ldots du_{n}
\end{math}
is the Euclidean volume element, and
\begin{equation} \label{e:sn}
  |S^{n}| = \frac{2 \, \pi^{\frac{n+1}{2}}}{\Gamma \left( \frac{n+1}{2} \right)}
\end{equation}
is the well-known expression for the area of the unit sphere
\begin{math}
  S^{n} \subset \r{n+1}.
\end{math}
On occasion we will refer to the hyperplane \( \partial \h{n+1} = \r{n} \) where the half-spheres are centered as the \emph{centerset}.
\par
The somewhat cumbersome choice of parametrization in Equation~\eqref{e:s} is intentional: in fact, it is key to our strategy.  Following Norton in~\cite{Norton-1980a}, let us introduce the mappings \( \n: f \mapsto F \) and \( \m: g \mapsto G \) by
\begin{eqnarray}
    F(x,y) &=&
    \begin{cases}
      y^{-\frac{1}{2}} \, f(x,\sqrt{y}), &y > 0, \\
      0 &y \leq 0,
    \end{cases} \label{e:n} \\
  G(x,y) &=&
    \begin{cases}
      \frac{1}{2} |S^{n}| \, y^{\frac{n-1}{2}} \, g(x,\sqrt{y}), &y > 0, \\
      0 &y \leq 0.
    \end{cases} \label{e:m}
\end{eqnarray}
Together Equations~\eqref{e:s}, \eqref{e:n} and~\eqref{e:m} imply the relation
\begin{equation} \label{e:p}
  G(x,y) =
  \int_{\r{n}}
    F(x + u, y - |u|^{2}) \,
  du,
\end{equation}
which we adopt as the definition of operator~$\norton$.  Furthermore, since the lower-case functions~$(f,g)$ are uniquely determined by the upper-case functions~$(F,G)$ via the inverse mappings \( \n^{-1}: F \mapsto f \) and \( \m^{-1}: G \mapsto g \) given by
\begin{eqnarray}
  f(x,y) &=& y \, F(x,y^{2}) \mid_{y>0}, \label{e:in} \\
  g(x,y) &=& \frac{2}{|S^{n}|} \, y^{1-n} \, G(x,y^{2}) \mid_{y>0}. \label{e:im}
\end{eqnarray}
Thus the spherical mean transform~\eqref{e:s} and the convolution equation~\eqref{e:p} are equivalent relations.
\par
The operational meaning of Equations~\eqref{e:s}, \eqref{e:n}, \eqref{e:m}, and~\eqref{e:p} can be conveniently illustrated by the following commutative diagram:
\begin{equation} \label{e:cd}
  \begin{CD}
    f \in \sh{n+1}         @>\sonar>>         g \in \sonar \left( \sh{n+1} \right)  \\
    @V{\n}VV                                   @V{\m}VV\\
    F \in \sr{n+1}         @>\norton>>        G \in \norton \left( \sr{n+1} \right)
  \end{CD}
\end{equation}
For our purposes it is not necessary to characterize the spaces~$\sonar \left( \sh{n+1} \right)$ and~$\norton \left( \sr{n+1} \right)$.  However, it will be useful to discuss the mapping properties of the convolution operator~$\norton$ on the space of compactly supported functions~$\ccr{n+1}$; we do that in Section~\ref{s:range}.
\par
As can be seen from Diagram~\ref{e:cd}, the operators~$\sonar$ and~$\norton$ are intertwined by the mappings~$\n$ and~$\m$.  Consequently, inverting the operator~$\sonar$ is equivalent to inverting the convolution operator~$\norton$:
\begin{equation} \label{e:is_norton}
  \sonar^{-1} = \n^{-1} \circ \norton^{-1} \circ \m.
\end{equation}
\par
We conclude this section by defining the adjoints~$\norton^{*}$,~$\sonar^{*}$ and a few additional operators needed to state our results in Section~\ref{s:results}.
\par
As follows from Equation~\eqref{e:p}, the $L^{2}$-adjoint~$\norton^{*}$ is given by
\begin{equation} \label{e:p*}
  \norton^{*}(G)(x,y) =
  \int_{\r{n}}
    G(x + u, y + |u|^{2}) \,
  du.
\end{equation}
In defining~$\sonar^{*}$, however, we are going to follow~\cite{Andersson-1988} and~\cite{Klein-2003}\footnote{Norton does not use the adjoint spherical mean transform in~\cite{Norton-1980a}; nor do Xu and Wang in~\cite{Xu-Wang-2005}. Fawcett defines an ``approximate inversion'' operation which is quite similar to~\eqref{e:s*}.  The adjoint~$\sonar^{*}$ used by Nessibi, Trimeche and Rachdi in~\cite{Nessibi-et-al-1995} differs from ours by a constant.}:
\begin{equation} \label{e:s*}
  \sonar^{*}(g)(x,y) =
  \int_{\r{n}}
    g(x + u, \sqrt{y^{2} + |u|^{2}}) \,
  du.
\end{equation}
Finally, for a rapidly decreasing function \( \phi(x,y) \in \sr{n+1} \), we define the Hilbert transform~$\hilbert_{y}(\phi)(x,y)$, acting on the $y$-variable, and a related operator~$\Lambda_{y}$, by
\begin{eqnarray}
  \hilbert_{y}(\phi)(x,y) &=&
    \frac{1}{\pi} \, \pv \,
    \int\limits_{-\infty}^{+\infty}
      \frac{\phi(x,t)}{y - t} \,
    dt,  \label{e:hy} \\
  \Lambda_{y} &=&
    \begin{cases}
      (-1)^{\frac{n-1}{2}} \, (\hilbert_{y} \circ \partial_{y}^{n}), &n=1,3,5,\ldots, \\
      (-1)^{\frac{n}{2}} \, \partial_{y}^{n} &n=2,4,6,\ldots
    \end{cases}  \label{e:lambda}
\end{eqnarray}
In Equations~\eqref{e:hy} and~\eqref{e:lambda} the symbol~$\partial_{y}$ denotes partial differentiation and~$\pv$ stands for Cauchy principal value.

\section{Main results}
\label{s:results}

Our main result---the inversion of the spherical mean transform~$\sonar$---follows from Theorem~\ref{t:ip} stated below.

\begin{theorem} \label{t:ip}
The following inversion formula holds:
\begin{equation} \label{e:ip}
  F = \pi^{-n} \, (\norton^{*} \circ \Lambda_{y} \circ \norton)(F), \quad F \in \sr{n+1}.
\end{equation}
\end{theorem}

Equation~\eqref{e:ip} will be derived in Section~\ref{s:fourier} using Fourier analysis and the theory of distributions.  However, the final proof of the validity of Theorem~\ref{t:ip} will be deferred until Section~\ref{s:radon} where it will follow from the classical Radon inversion formula.  Meanwhile, we state and prove our main result.

\begin{theorem} \label{t:is}
The following inversion formula holds:
\begin{equation} \label{e:is}
  f = y \, (\sonar^{*} \circ \k_{y} \circ \sonar)(f), \quad f \in \sh{n+1},
\end{equation}
where the one-dimensional operator~$\k_{y}$ (acting on the $y$-variable) is defined by
\begin{equation} \label{e:ky}
  \k_{y}(\phi)(y) =  \frac{|S^{n}|}{2 \, \pi^{n}} \,
  \Lambda_{t} \left( t^{\frac{n-1}{2}} \, \phi(\sqrt{t}) \right) \bigg|_{t = y^{2}}
\end{equation}
\end{theorem}

\begin{proof}
Combining Equations~\eqref{e:is_norton} and~\eqref{e:ip}, we obtain
\[
  f = \pi^{-n} \, (\n^{-1} \circ \norton^{*} \circ \Lambda_{y} \circ \m \circ \sonar)(f).
\]
We now define the mapping~$\l$ as a simple modification of Equation~\eqref{e:n}:
\begin{equation} \label{e:l}
    \l(f)(x,y) =
    \begin{cases}
      f(x,\sqrt{y}), &y > 0, \\
      0 &y \leq 0.
    \end{cases}
\end{equation}
Notice that the inverse mapping~$\l^{-1}$ is given by
\begin{equation} \label{e:il}
  f(x,y) = F(x,y^{2}) \mid_{y>0},
\end{equation}
which is analogous to Equation~\eqref{e:in}.

From Equations~\eqref{e:p*}, \eqref{e:s*}, and \eqref{e:l} follows the intertwining relation
\begin{equation} \label{e:sp*}
  y \, \sonar^{*} = \n^{-1} \circ \norton^{*} \circ \l.
\end{equation}
To apply Equation~\eqref{e:sp*}, we insert the composition \( \l \circ \l^{-1} \) into Equation~\eqref{e:is} as follows:
\begin{equation} \label{e:is_modified}
  f = \pi^{-n} \, (\n^{-1} \circ \norton^{*} \circ \l
  \circ \l^{-1} \circ \Lambda_{y} \circ \m \circ \sonar)(f).
\end{equation}
Although \( \l \circ \l^{-1} \) does not act as an identity, the validity of~\eqref{e:is} is unaffected because the values of~$\norton^{*}(F)(x,y)$ for \( y > 0 \) depend only on the restriction of~$F(x,y)$ to the upper half-space.  Using the intertwining relation~\eqref{e:sp*}, we replace the triple composition
\(
  (\n^{-1} \circ \norton^{*} \circ \l)
\)
with \( y \, \sonar^{*}. \)  Finally, by explicitly rewriting the composition \( (\l^{-1} \circ \Lambda_{y} \circ \m), \) using Equations~\eqref{e:l} and~\eqref{e:m}, we are led to the operator~$\k_{y}$ defined by Equation~\eqref{e:ky}.
\end{proof}

Given the practical importance of spherical means in low dimensions, we conclude this section with explicit inversion formulas for~$\sonar$ in dimensions two and three.

\begin{corollary} \label{c:is23}
Let \( g = \sonar(f) \) where~$f$ is rapidly decreasing.  The following inversion formula holds in~$\h{2}$:
\begin{equation} \label{e:is2}
  f(x,y) = \frac{y}{\pi} \,
  \int_{-\infty}^{\infty}
  \left[
    \pv \,
    \int_{0}^{\infty}
      \frac{g'_{y}(x + u,s)}{y^{2} + u^{2} - s^{2}} \,
    ds
  \right] \,
  du,
\end{equation}
where the bracketed integral is interpreted in the sense of Cauchy principal value.

In~$\h{3}$ the analogue of Equation~\eqref{e:is2} can be written as
\begin{equation} \label{e:is3}
  \begin{split}
    f(x,y) &= \frac{y}{2 \pi} \,
      \int_{\r{2}}
      \left(
        \frac{g(x+u,\sqrt{y^{2} + |u|^{2}})}{\left( y^{2} + |u|^{2} \right)^{\frac{3}{2}}}
      \right.
      \\&-
        \frac{g'_{y}(x+u,\sqrt{y^{2} + |u|^{2}})}{y^{2} + |u|^{2}}
      \\&-
      \left.
        \frac{g''_{yy}(x+u,\sqrt{y^{2} + |u|^{2}})}{\sqrt{y^{2} + |u|^{2}}}
      \right)  \,
    du,
  \end{split}
\end{equation}
where now \( x = (x_{1},x_{2}). \)
\end{corollary}
\begin{proof}
Differentiating Equation~\eqref{e:m} (with $n=2$) with respect to the $y$-variable, we obtain
\begin{displaymath}
  (\partial_{y}G)(x,y) =
    G_{y}(x,y) =
    \begin{cases}
      \pi \, g_{y}(x,\sqrt{y}) \, \frac{1}{2 \, \sqrt{y}}, &y > 0, \\
      0, &y \leq 0.
    \end{cases}
\end{displaymath}
Now the successive application of Hilbert transform~\eqref{e:hy} leads to
\begin{displaymath}
    (\hilbert_{y} G_{y}) (x,y)
  =
    \pv \int\limits_{0}^{\infty}
      \frac{g_{y}(x,\sqrt{t})}{y - t} \,
      \frac{dt}{2 \,\sqrt{t}} =
    \pv \int\limits_{0}^{\infty}
      \frac{g_{y}(x,s)}{y - s^{2}} \,
      ds.
\end{displaymath}
Finally, applying the adjoint~$\norton^{*}$ (see Equation~\eqref{e:p*}), we get the following explicit relation between~$F$ and~$g$
\begin{displaymath}
  F(x,y) =
  \frac{1}{\pi} \,
  \int\limits_{-\infty}^{\infty}
  \left[
    \pv \int\limits_{0}^{\infty}
      \frac{g_{y}(x+v,s)}{y + v^{2} - s^{2}} \,
      ds
  \right] \,
  dv,
\end{displaymath}
from which the inversion formula~\eqref{e:is2} results after the substitution
\begin{displaymath}
  f(x,y) = y \, F(x,y^{2}) \mid_{y>0},
\end{displaymath}
in accordance with Equation~\eqref{e:in}.

The argument leading to Equation~\eqref{e:is3} is completely analogous and will therefore be omitted.  The corollary is proved.
\end{proof}
\section{Mapping properties of~$\norton$}
\label{s:range}

In this section we examine the mapping properties of the convolution operator~$\norton$ on the space of smooth compactly supported functions.  The lemmas below do not fully characterize the range of~$\norton$, yet they establish some useful facts that will guide us to the correct expression for~$\norton^{-1}$ in Sections~\ref{s:fourier} and~\ref{s:radon}.  We begin with a simple asymptotic estimate.

\begin{lemma} \label{l:asymptotic}
Let~$F(x,y)$ be a smooth compactly supported function in~$\r{n+1}$ $(n \geq 1)$ with support lying in the upper half-space $\h{n+1}$.  Fix \( (x_{0},y_{0}) \in \h{n+1} \) and consider
\[
  G(v) =
  \left|
    \norton(F)(x_{0} + v, y_{0} + |v|^{2})
  \right|,
\]
as a function of \( v \in \r{n}. \)  There exist~$r > 0$ and~$h > 0$ (possibly dependent on~$(x_{0},y_{0})$) such that
\begin{equation} \label{e:asymptotic}
  \lim_{|v| \to \infty} |v| \, G(v) \leq
  \frac{\max |F|}{2} \, |B^{n-1}| \, r^{n-1} \, h,
\end{equation}
where~$|B^{n}|$ denotes the volume of a unit ball in~$\r{n}$ \( (|B^{0}| = 1). \)
\end{lemma}

\begin{proof}
It is instructive to first prove the statement of the lemma for \( n = 1. \)
\par
Since~$F$ is compactly supported, we can choose~$r>0$ and~$h>0$ so that~$\supp(F)$ is entirely contained within the rectangle
\(
  [x_{0}-r,x_{0}+r] \times [0,h].
\)

\begin{figure}[ht!]
  \centering
  \scalebox{1}{\includegraphics{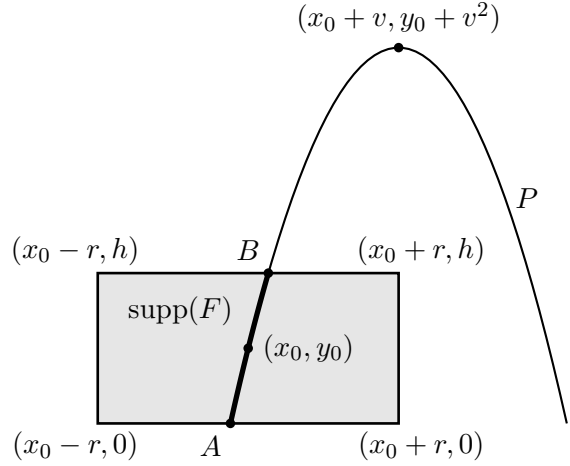}}
  \caption{Estimation of~$G(v)$ in two dimensions}
  \label{fig:asymptotic2d}
\end{figure}

As~$v$ goes to infinity, the parabola
\[
  P = \{ (x_{0} + v + u, y_{0} + v^{2} - u^{2}) \mid u \in \r{1} \},
\]
passing through~$(x_{0},y_{0})$, will intersect only the horizontal sides of the rectangle at points $A(\sqrt{y_{0} + v^{2}},0)$ and $B(\sqrt{y_{0} + v^{2}-h},h)$, as shown in Figure~\ref{fig:asymptotic2d}.  Therefore
\[
  \begin{split}
  G(v) &=
    \left|
      \int_{-\infty}^{+\infty}
        F(x_{0} + v + u, y_{0} + v^{2} - u^{2}) \,
      du
    \right|\\
  &=
    \left|
      \int_{-\sqrt{y_{0} + v^{2}}}^{-\sqrt{y_{0} + v^{2}-h}}
        F(x_{0} + v + u, y_{0} + v^{2} - u^{2}) \,
      du
    \right|\\
  &\leq
  \max |F| \,
  \left(
    \sqrt{y_{0} + v^{2}} - \sqrt{y_{0} + v^{2}-h}
  \right).
  \end{split}
\]
It is now a simple matter to compute
\[
   \lim_{v \to \infty} v \, G(v) = \frac{1}{2} \, \max |F| \, h.
\]
Notice that this is in agreement with Equation~\eqref{e:asymptotic} for \( n = 1. \)
\par
We now turn our attention to the case \( n > 1. \)  Our strategy is going to be the same as in the preceding case, except that now the support of~$F$ will be confined to the interior of the cylinder
\(
  C = B^{n}(x_{0},r) \times [0,h]
\)
with sufficiently large dimensions.  In order to visualize the intersection of~$C$ with the paraboloid
\[
  P = \{ (x_{0} + v + u, y_{0} + |v|^{2} - |u|^{2}) \mid u \in \r{n} \}
\]
we project both orthogonally onto~$\r{n}$ using the mapping
\(
  \pi: (x,y) \in \h{n+1} \mapsto x \in \r{n}.
\)
The result is shown in Figure~\ref{fig:asymptotic} below.

\begin{figure}[ht]
  \centering
  \scalebox{1}{\includegraphics{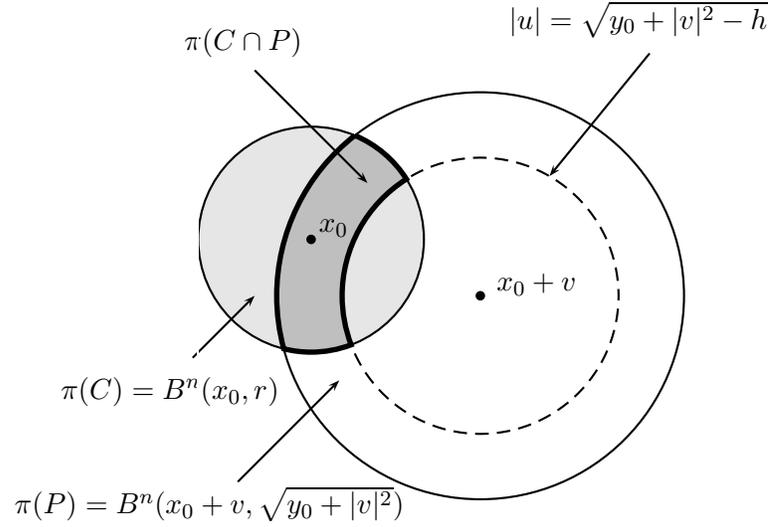}}
  \caption{Estimation of~$G(v)$ in higher dimensions}
  \label{fig:asymptotic}
\end{figure}

The lightly-shaded disk is the projection of the cylinder~$C$. The large solid circle is the intersection of the paraboloid~$P$ with the plane~$y = 0$ containing the base of~$C$.  The dashed circle is the intersection of~$P$ with the plane~$y = h$ containing the top of the cylinder~$C$.  Of primary interest to us is the darkly-shaded quadrilateral bound by four circular arcs (representing spherical surfaces in~$\r{n})$.  This region is the orthogonal projection~$\pi(C \cap P)$ which we now use to derive the following estimate:
\[
  \begin{split}
    G(v)
&=
    \left|
      \int_{\r{n}} F(x_{0} + v + u,y_{0} + |v|^{2} - |u|^{2}) \, du
    \right| \\
&=
    \left|
      \int_{\pi(C \cap P)} F(x_{0} + v + u,y_{0} + |v|^{2} - |u|^{2}) \, du
    \right| \\
&\leq
    \max |F| \, \vol(\pi(C \cap P)).
  \end{split}
\]
To finish the proof we will find the volume of~$\pi(C \cap P)$.  To this end, we express~$\vol(\pi(C \cap P))$ as the difference of volumes of two lens-shaped domains:
\begin{equation} \label{e:volCP}
  \begin{split}
    \vol(\pi(C \cap P))
  &=
      \vol \left( B^{n}(x_{0},r) \cap B^{n}(x_{0}+v,\sqrt{y_{0}+|v|^{2}}) \right) \\
  &-
      \vol \left(B^{n} (x_{0},r) \cap B^{n}(x_{0}+v,\sqrt{y_{0}+|v|^{2}-h}) \right)\\
  &= r^{n} \,
    \left[
      U \left( \frac{r^{2} - y_{0}}{2 \, r \, |v|} \right) -
      U \left( \frac{r^{2} - y_{0} + h}{2 \, r \, |v|} \right)
    \right] \\
  &+
    (y_{0} + |v|^{2})^{\frac{n}{2}} \,
    U \left( \frac{2 \, |v|^{2} + y_{0} - r^{2}}{2 \, |v| \, \sqrt{y_{0} + |v|^{2}}} \right)\\
  &-
    (y_{0} + |v|^{2} - h)^{\frac{n}{2}} \,
    U \left( \frac{2 \, |v|^{2} + y_{0} - r^{2} - h}{2 \, |v| \, \sqrt{y_{0} + |v|^{2} - h}} \right),
  \end{split}
\end{equation}
where
\begin{equation} \label{e:U}
  U(s) = |B^{n-1}| \,
  \int_{s}^{1}
    (1 - t^{2})^{\frac{n-1}{2}} \,
  dt.
\end{equation}
The form of Equation~\eqref{e:volCP} suggests the introduction of the following limits:
\begin{eqnarray*}
  L_{1} &=& r^{n} \, \lim_{|v| \to \infty} |v| \,
    \left[
      U \left( \frac{r^{2} - y_{0}}{2 \, r \, |v|} \right) -
      U \left( \frac{r^{2} - y_{0} + h}{2 \, r \, |v|} \right)
    \right], \\
  L_{2} &=& \lim_{|v| \to \infty} |v| \, (y_{0} + |v|^{2})^{\frac{n}{2}} \,
    U \left( \frac{2 \, |v|^{2} + y_{0} - r^{2}}{2 \, |v| \, \sqrt{y_{0} + |v|^{2}}} \right), \\
  L_{3} &=& \lim_{|v| \to \infty} |v| \,
    (y_{0} + |v|^{2} - h)^{\frac{n}{2}} \,
    U \left( \frac{2 \, |v|^{2} + y_{0} - r^{2} - h}{2 \, |v| \, \sqrt{y_{0} + |v|^{2} - h}} \right).
\end{eqnarray*}
The required limit on the left-hand side of Equation~\eqref{e:asymptotic} is now a simple linear combination of~$L_{1}$, $L_{2}$ and $L_{3}$:
\[
  \lim_{|v| \to \infty} |v| \, G(v) = \max |F| \, (L_{1} + L_{2} - L_{3}).
\]
To find $L_{1}$, substitute \( |v|^{-1} = \epsilon. \)  This leads to
\[
  L_{1} = r^{n} \, \lim_{\epsilon \to 0^{+}}
  \frac{
      U \left( \frac{r^{2} - y_{0}}{2 \, r} \, \epsilon \right) -
      U \left( \frac{r^{2} - y_{0} + h}{2 \, r} \, \epsilon \right)
  }{\epsilon} = \frac{1}{2} \, |B^{n-1}| \, r^{n-1} \, h,
\]
since \( U'(0) = -|B^{n-1}|. \)

Expanding~$U(s)$, given by Equation~\eqref{e:U}, into a generalized series at \( s = 1, \) we obtain
\[
  U(s) = |B^{n-1}| \frac{2^{\frac{n+1}{2}}}{n+1} \, (1 - s)^{\frac{n+1}{2}} +
    \bigo \left( (1 - s)^{\frac{n+3}{2}} \right),
\]
which, together with
\[
  \frac{2 \, |v|^{2} + y_{0} - r^{2}}{2 \, |v| \, \sqrt{y_{0} + |v|^{2}}} =
  1 - \frac{1}{2} \, r^{2} \, |v|^{-2} + \bigo(|v|^{-4}),
\]
leads to the values of $L_{2}$ and $L_{3}$:
\[
  L_{2} = L_{3} =  |B^{n-1}| \, \frac{2^{\frac{n-1}{2}}}{n+1} \, r^{n+1}.
\]
We conclude that, for \( n > 1, \)
\[
    \lim_{|v| \to \infty} |v| \, G(v) = \max |F| \, L_{1} =
    \frac{\max |F|}{2} \, |B^{n-1}| \, r^{n-1} \, h,
\]
as required.
\end{proof}

\begin{remark} \label{rem:asymptotic}
From Lemma~\ref{l:asymptotic} follows that the composition
\(
  (\norton^{*} \circ \norton)
\)
cannot be freely applied to functions in~$\ccr{n+1}$.  Indeed, let \( F \in \ccr{n+1} \) be nonnegative and let \( (x_{0},y_{0}) \in \supp(F). \)  Then
\[
  (\norton^{*} \circ \norton)(F)(x_{0},y_{0}) =
  \int_{\r{n}}
    \norton(F)(x_{0} + v, y_{0} + |v|^{2}) \,
  dv \to \infty,
\]
since
\(
  \norton(F)(x_{0} + v, y_{0} + |v|^{2}) = \bigo(|v|^{-1}).
\)
Accordingly, any inversion formula of the type
\(
  F = (\mathcal{X} \circ \norton^{*} \circ \norton)(F)
\)
requires strong restrictions on the functional domain of~$\norton$ to assure the convergence of~$(\norton^{*} \circ \norton)(F)$.
\par
The above analysis applies, with minor modifications, to the spherical mean transform~$\sonar$.  As Klein shows in~\cite{Klein-2003}, the inversion formulas of the type
\(
  f = (\mathcal{X} \circ \sonar^{*} \circ \sonar)(f),
\)
derived in~\cite{Fawcett-1985} and~\cite{Nessibi-et-al-1995}, can be interpreted conventionally only on a limited subspace of~$\cch{n+1}$.  To avoid convergence issues associated with the \emph{backprojection} operation $\sonar^{*} \circ \sonar$, Klein suggests to precede the adjoint~$\sonar^{*}$ with a derivative~$\partial_{y}$.  Our solution, on the other hand, is to seek an inversion formula of the kind~\( f = (\sonar^{*} \circ K \circ \sonar)(f). \)
\end{remark}

Our next lemma characterizes the shape of the support of \( G = \norton(F); \) It will be particularly useful in Section~\ref{s:radon}.

\begin{lemma} \label{l:compact}
Let
\(
  G = \norton(F)
\)
for some
\(
  F \in \ccr{n+1}
\)
\(
  (n \geq 1)
\)
and let~$\Pi$ be any hyperplane in~$\r{n+1}$ that meets the support of~$G$ in general position.  The intersection
\(
  \supp(G) \cap \Pi
\)
is compact.
\end{lemma}

\begin{proof}
Figure \ref{fig:compact} illustrates the lemma in two dimensions \( (n = 1). \)

\begin{figure}[ht]
  \centering
  \scalebox{1}{\includegraphics{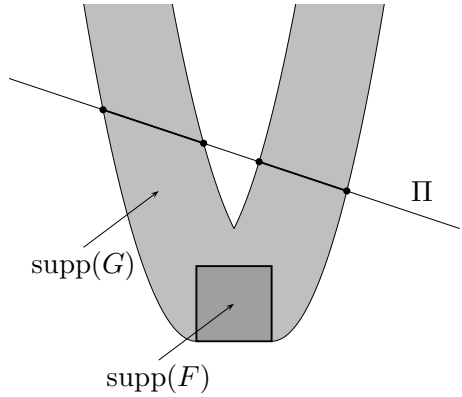}}
  \caption{Intersection $\supp(G) \cap \Pi$ in two dimensions}
  \label{fig:compact}
\end{figure}

The dark square, say, $[-a,a]^{2}$ for some large \( a > 0, \) bounds the support of~$F$:
\(
  \supp(F) \subset [-a,a]^{2}.
\)
The lightly-shaded $U$-shaped region is the infinite union of parabolas covering~$\supp(G)$:
\[
  \supp(G) \subset \bigcup_{(x,y) \in [-a,a]^{2}} \{ (x + v, y + v^{2}) \mid v \in \r{1} \}.
\]
Evidently, the intersection of the $U$-shaped region with a generic line~$\Pi$ is compact.  In fact, in two dimensions, \emph{any} line meeting $\supp(G)$ forms a compact intersection.
\par
We now present a rigorous proof of the general case which we base on Heine-Borel Theorem.
\par
Since both~$\Pi$ and~$\supp(G)$ are relatively closed in~$\r{n+1}$, the intersection \( \supp(G) \cap \Pi \) is closed.  Our task therefore is reduced to showing that \( \supp(G) \cap \Pi \) is bounded.
\par
Let
\(
  \pi: (x,y) \in \h{n+1} \mapsto x \in \r{n}
\)
denote the orthogonal projection, as in Lemma~\ref{l:asymptotic}.  Since the hyperplane~$\Pi$ is in general position, the boundedness of \( \supp(G) \cap \Pi \) implies and is implied by the boundedness of the projection~\( \pi(\supp(G) \cap \Pi). \) We therefore direct our attention to the latter.
\par
To visualize $\pi(\supp(G) \cap \Pi)$ we describe~$\Pi$ explicitly with a standard equation \( \omega \cdot x + \theta \, y = p. \)  Our assumption about~$\Pi$ being in general position allows us to disregard the special cases \( \omega = 0 \) and \( \theta = 0. \) Hence we can normalize \( \omega \in \r{n} \) to have unit length and choose \( \theta \in \r{1} \) to be positive.
\par
Now fix \( (x_{0},y_{0}) \in \supp(F) \) and consider the intersection of~$\Pi$ with the paraboloid of revolution:
\[
  P = \{ (x_{0} + v, y_{0} + |v|^{2}) \mid v \in \r{n} \}.
\]
The projection \( \pi(P \cap \Pi) \) is readily seen to be a sphere with radius
\[
  r = \frac{\sqrt{1 + 4 \, \theta \, (p - \omega \cdot x_{0} - \theta \, y_{0})}}{2 \, \theta}
\]
and center \( x_{0} + \omega / (2 \, \theta) \in \r{n}. \)  Set
\[
  r_{\max} = \frac{1}{2 \theta} \, \sqrt{1 + 4 \, \theta \, \max_{(x,y) \in \supp(F)} |p - \omega \cdot x - \theta \, y|}.
\]
We have the inclusion:
\[
  \pi(P \cap \Pi) \subseteq B^{n}\left( x_{0} + \frac{\omega}{2 \, \theta}, r_{\max} \right).
\]
Since~$\supp(G)$ is covered by paraboloids
\[
  \supp(G) \subseteq \bigcup_{(x,y) \in \supp(F)} \{ (x + v, y + |v|^{2}) \mid v \in \r{n} \},
\]
the projection~$\pi(\supp(G) \cap \Pi)$ is covered by balls:
\[
  \pi(\supp(G) \cap \Pi) \subseteq \bigcup_{x \in \pi(\supp(F))}
  B^{n} \left( x + \frac{\omega}{2 \, \theta}, r_{\max} \right).
\]
We conclude that, since~$\pi(\supp(F))$ is compact and~$r_{\max}$ depends only on~$(\omega,\theta)$, the projection~$\pi(\supp(G) \cap \Pi)$ and, consequently, $\supp(G) \cap \Pi$ must be bounded.  This proves the lemma.
\end{proof}

\begin{remark} \label{rem:compact1}
Figure~\ref{fig:compact} shows that the function \( G(x,y) = \norton(F)(x,y) \) (in two dimensions) is \emph{not} bandlimited in the $y$-variable, as assumed by Norton in~\cite{Norton-1980a}.  Indeed, as a function of the $y$-variable alone, $G(x,y)$ is compactly supported.  Yet, as shown on Page~$360$ in~\cite{Gasquet-Witomski-1999}, a compactly supported function \emph{cannot} be bandlimited for its Fourier transform is analytic and cannot vanish outside a compact set.  Accordingly, in our derivation of the inversion formula for~$\norton$, we avoided any assumptions about the spectrum of~$G(x,y)$.  As a result, our formula for~$\norton^{-1}$ is quite different from Norton's.
\end{remark}


\section{Fourier inversion of the convolution~$\norton$}
\label{s:fourier}

Our aim in this section is to derive Equation~\eqref{e:ip} which gives the inversion formula for the convolution operator~$\norton$; recall that the latter is defined by Equation~\eqref{e:p} in Section~\ref{s:preliminaries}.  The derivation largely amounts to straightforward de-convolution of Equation~\eqref{e:p} which requires us to introduce Fourier notation and collect a few calculus facts.
\par
Let $\phi(x,y)$ be an integrable function in~$\r{n+1}$.  We define the Fourier transform~$\widehat{\phi}(\xi,\eta)$ by
\begin{equation} \label{e:fourier}
  \widehat{\phi}(\xi,\eta) =
  \int_{\r{n}} \int_{\r{1}}
    \exp(-i \, \xi \cdot x - i \, \eta \, y) \,
    \phi(x,y) \,
  dy \, dx.
\end{equation}

\begin{remark} \label{rem:hat}
The ``hat''-symbol is often used to denote the classical Radon transform which will be introduced in Section~\ref{s:radon}.  To avoid any ambiguity, we henceforth agree to use the ``hat'' only as a shorthand for the classical Fourier transform in~$\r{n+1}$.
\end{remark}

To avoid explicit use of the Fourier inversion formula, we will need to recognize Fourier transforms of the operators introduced in Section~\ref{s:preliminaries}; these we now present below.
\par
As follows from Equation~(33.4) on Page~$313$ in~\cite{Gasquet-Witomski-1999}, with our choice of the constant in front of the integral in Equation~\eqref{e:hy} the Fourier symbol of the Hilbert transform~$\hilbert_{y}$ is given by
\begin{equation} \label{e:hy_hat}
  \widehat{\hilbert_{y}} =
  -i \, \sgn(\eta) =
  \begin{cases}
    -i  & \eta \geq 0, \\
    +i  & \eta < 0,
  \end{cases}
\end{equation}
that is,
\begin{math}
  \widehat{\hilbert_{y}(\phi)} = \widehat{\hilbert_{y}} \, \widehat{\phi} =-i \, \sgn(\eta) \, \widehat{\phi}.
\end{math}

Proposition~(17.2.1) on Page~$157$ in~\cite{Gasquet-Witomski-1999} shows\footnote{Gasquet and Witomski use the symmetric definition of the Fourier transform with a factor of~$2\pi$ in the exponential} that the Fourier symbol of the partial derivative
\begin{math}
  \frac{\partial}{\partial y} = \partial_{y}
\end{math}
is given by
\begin{math}
  \widehat{\partial_{y}} = i \, \eta;
\end{math}
(Of course, this simple result can be easily derived from scratch using integration by parts.)  Together with Equation~\eqref{e:hy_hat}, the formula for~$\widehat{\partial_{y}}$ leads to the Fourier symbol of the $\Lambda$-operator:
\begin{equation} \label{e:lambda_hat}
  \widehat{\Lambda_{y}} = |\eta|^{n}, \quad n = 1,2,3,\ldots
\end{equation}
The latter is a familiar fact in signal processing and tomography.
\par
It now remains to establish some calculus facts which we prefer to keep separate from Fourier manipulations.  We shall require the following integrals found in~\cite{Gradshteyn-Ryzhik-1980} (Page~$490$ Equations~(5) and~(6)):
\begin{eqnarray}
  \int\limits_{0}^{\infty}
    x^{\mu-1} \, \exp(-\beta \, x) \, \sin(\delta x) \,
  dx
    &=&
  \frac{\Gamma(\mu) \, \sin \left( \mu \, \arctan \frac{\delta}{\beta} \right)}{(\beta^{2} + \delta^{2})^{\frac{\mu}{2}}}, \label{e:gr1} \\
  \int\limits_{0}^{\infty}
    x^{\mu-1} \, \exp(-\beta \, x) \, \cos(\delta x) \,
  dx
    &=&
  \frac{\Gamma(\mu) \, \cos \left( \mu \, \arctan \frac{\delta}{\beta} \right)}{(\beta^{2} + \delta^{2})^{\frac{\mu}{2}}}. \label{e:gr2}
\end{eqnarray}
Integrals~\eqref{e:gr1} and~\eqref{e:gr2} converge for
\begin{math}
  \Re(\mu) > 0,
\end{math}
\begin{math}
  \Re(\beta) > |\Im(\delta)|
\end{math}
and are needed in the proof of Lemma~\ref{l:phat} below.
\begin{lemma} \label{l:phat}
For any
\begin{math}
  \xi \in \r{n}
\end{math}
and
\begin{math}
  \eta \neq 0,
\end{math}
\begin{equation} \label{e:phat}
  \begin{split}
    \lim_{\epsilon \to 0^{+}}
  &
    \int_{\r{n}}
      \exp \left( i \, \xi \cdot x - ( \epsilon + i \, \eta) \, |x|^{2} \right) \,
    dx \\
  &=
    \left( \frac{\pi}{|\eta|} \right)^{\frac{n}{2}} \,
    \exp \left( i \, \frac{|\xi|^{2} - \pi \, n \, |\eta|}{4 \, \eta} \right).
  \end{split}
\end{equation}
\end{lemma}
\begin{proof}
To find the limit in Equation~\eqref{e:phat}, we first compute the integral using the following steps.
\par
Completing the square, we obtain
\begin{displaymath}
  \exp \left( -\frac{|\xi|^{2}}{4 \, ( \epsilon + i \, \eta)} \right) \,
  \int_{\r{n}}
    \exp \left(-(\epsilon + i \, \eta) \, \left| x - \frac{i \, \xi}{2 \, (\epsilon + i \, \eta)} \right|^{2} \right) \,
  dx,
\end{displaymath}
which upon substitution
\begin{displaymath}
  x - \frac{i \, \xi}{2 \, (\epsilon + i \, \eta)} = u,
\end{displaymath}
results in
\begin{displaymath}
  \exp \left( -\frac{|\xi|^{2}}{4 \, ( \epsilon + i \, \eta)} \right) \,
  \int_{\r{n}}
    e^{-(\epsilon + i \, \eta) \, | u |^{2}}
  du.
\end{displaymath}
We now switch to polar coordinates and integrate over the angular variable thereby obtaining
\begin{displaymath}
  |S^{n-1}| \,
  \exp \left( -\frac{|\xi|^{2}}{4 \, ( \epsilon + i \, \eta)} \right) \,
  \int_{0}^{\infty}
    \exp \left(-(\epsilon + i \, \eta) \, r^{2} \right) \,
    r^{n-1} \,
  dr,
\end{displaymath}
where~$|S^{n-1}|$ is the constant given by Equation~\eqref{e:sn}.  Now substitute
\begin{math}
  r = \sqrt{s}
\end{math}
and apply Euler's formula
\begin{math}
  \exp (i \, \theta) = \cos (\theta) + i \, \sin (\theta)
\end{math}
to rewrite the integral as
\begin{displaymath}
  \begin{split}
    \frac{|S^{n-1}|}{2}
  &
    \exp \left( -\frac{|\xi|^{2}}{4 \, ( \epsilon + i \, \eta)} \right)
    \left(
      \int_{0}^{\infty}
        \exp(-\epsilon \, s) \,
        \cos(\eta \, s) \,
        s^{\frac{n-2}{2}} \,
      ds
    \right.\\
  &
    \left.-i \,
    \int_{0}^{\infty}
      \exp(-\epsilon \, s) \,
      \sin(\eta \, s) \,
      s^{\frac{n-2}{2}} \,
    ds
    \right)
  \end{split}
\end{displaymath}
Notice that the complex integral in Equation~\eqref{e:phat} is now reduced to two real-valued integrals similar to the ones listed in Equations~\eqref{e:gr1} and~\eqref{e:gr2}.  Setting
\begin{math}
  \beta = \epsilon,
\end{math}
\begin{math}
  \delta = \eta
\end{math}
and
\begin{math}
  \mu = n/2,
\end{math}
in tabular integrals~\eqref{e:gr1} and~\eqref{e:gr2}, and substituting the results into the preceding expression, we are led to
\begin{displaymath}
    \frac{|S^{n-1}| \, \Gamma \left( \frac{n}{2} \right)}{2} \,
    \left( \epsilon^{2} + \eta^{2} \right)^{-\frac{n}{4}} \,
    \exp
    \left(
      -\frac{|\xi|^{2}}{4 \, ( \epsilon + i \, \eta)} -
      \frac{n}{2} \, \tan^{-1} \left( \frac{\eta}{\epsilon} \right) \, i
    \right).
\end{displaymath}
Now the limit as \begin{math} \epsilon \to 0^{+} \end{math} can be easily taken and the subsequent replacement of the constant~$|S^{n-1}|$ with its value given by Equation~\eqref{e:sn} leads to Equation~\eqref{e:phat}, as required.
\end{proof}
We are now ready to derive the inversion formula~\eqref{e:ip}.
\begin{proof}[Derivation of Equation~\eqref{e:ip}]
Applying the Fourier transform~\eqref{e:fourier} to both sides of Equation~\eqref{e:p} and interchanging the order of integrations, we get the product relation
\begin{math}
  \widehat{G} = \widehat{\norton} \, \widehat{F}
\end{math}
in which the Fourier multiplier
\begin{displaymath}
  \widehat{\norton} =
  \left[
    \int_{\r{n}} \exp(i \, \xi \cdot u - i \, \eta \, |u|^{2}) \, du
  \right]
\end{displaymath}
is divergent for all
\begin{math}
  n > 2.
\end{math}
This necessitates a full switch to distributional perspective.
\par
Motivated by the form of Equation~\eqref{e:p}, we define a linear functional~$T$ on the Schwartz space of rapidly decreasing functions~$\sr{n+1}$ by
\begin{equation*}
  T (\phi) =
  \int_{\r{n}}
    \phi \left( x, -|x|^{2} \right) \,
  dx.
\end{equation*}
Evidently, $T$ is linear and continuous on~$\sr{n+1}$: as such, it is a tempered distribution in~$\sprimer{n+1}$ with a well-defined Fourier symbol
\begin{math}
  \widehat{T} \in \sprimer{n+1}.
\end{math}
\par
Next we reinterpret the right-hand side of Equation~\eqref{e:p} as the convolution of
\begin{math}
  T \in \sr{n+1}
\end{math}
with a \emph{test-function}
\begin{math}
  F \in \sr{n+1},
\end{math}
i.e.:
\begin{math}
  G = T \star F.
\end{math}
The result of application of the Fourier transform~\eqref{e:fourier} to Equation~\eqref{e:p} can now be restated as
\begin{math}
  \widehat{G} = \widehat{T} \, \widehat{F}.
\end{math}
Ostensibly, we have only replaced the symbol~$\norton$ with a new symbol~$T$.  Yet this subtle notational change makes all the difference.  Indeed, the Fourier symbol $\widehat{\norton}$, having lost classical meaning, can now be identified with~$\widehat{T}$---a well-defined tempered distribution in~$\sprimer{n+1}$.
\par
Following Kanwal in~\cite{Kanwal-2004}, we use an exponential \emph{convergence factor} to regularize~$\widehat{T}$, \footnote{The use of convergence factors had originated in mathematical physics as a heuristic way of computing ``finite parts'' of divergent integrals.  Yet, it is now well-known that the ``physics'' regularization of divergent Fourier integrals is equivalent to the formal definition of the Fourier transform on~$\sprimer{n}$ as long as the assumptions of the Lebesgue Dominated Convergence Theorem are satisfied.} i.e., we define
\begin{displaymath}
  \widehat{T} =
    \lim_{\epsilon \to 0^{+}}
    \int_{\r{n}}
      \exp \left( i \, \xi \cdot x - ( \epsilon + i \, \eta) \, |x|^{2} \right) \,
    dx,
\end{displaymath}
which is precisely the limit computed in Lemma~\ref{l:phat}.\footnote{Lemma~\ref{l:phat} (giving~$\widehat{T}$) is largely identical to Kanwal's computation of the fundamental solution of Shr\"odinger's equation in free space (Page~$270$ in~\cite{Kanwal-2004}). Kanwal further mentions that a faster way to compute the integral is to use the standard Fourier transform of a Gaussian.  We, however, we prefer the method used in this section as more informative.}
\par
Using Equation~\eqref{e:phat} we conclude that
\begin{equation*}
  \widehat{G}(\xi,\eta) =
  \left( \frac{\pi}{|\eta|} \right)^{\frac{n}{2}} \,
  \exp \left(-i \, \frac{|\xi|^{2} - \pi \, n \, |\eta|}{4 \, \eta} \right) \,
  \widehat{F}(\xi,\eta),
\end{equation*}
whence follows the reciprocal relation
\begin{math}
  \widehat{F} = (1/\widehat{T}) \, \widehat{G}
\end{math}
which we write in the form
\begin{displaymath}
  \widehat{F}(\xi,\eta) = \pi^{-n} \,
    \left[
      \left( \frac{\pi}{|\eta|} \right)^{\frac{n}{2}} \,
      \exp \left(-i \, \frac{|\xi|^{2} - \pi \, n \, |\eta|}{4 \, \eta} \right)
    \right] \,
    |\eta|^{n} \,
    \widehat{G}(\xi,\eta).
\end{displaymath}
It remains to identify the terms multiplying~$\widehat{G}$ on the right-hand side.  The bracketed term, being the complex conjugate~$\overline{\widehat{T}}$ corresponds to the adjoint~$T^{*}$ (which in turn corresponds to~$\norton^{*}$) while the term~$|\eta|^{n}$ is the Fourier symbol of the operator~$\Lambda_{y}$ given in Equation~\eqref{e:lambda_hat}.  We are thus led to Equation~\eqref{e:phat}, as required.
\end{proof}

\begin{remark} \label{rem:ordering}
It may appear that along with Equation~\eqref{e:ip} we may form other inversion formulas for the operator~$\norton$ by permuting~$\norton^{*}$ with the operators defining~$\Lambda_{y}$.  However, our discussion of the mapping properties of~$\norton$ in Section~\ref{s:range} prohibits~$\norton^{*}$ from being the leftmost operator (Remark~\ref{rem:asymptotic}).  It is quite possible that the inversion formula \( F = \pi^{-n} \, (\Lambda_{y} \circ \norton^{*} \circ \norton)(F) \) may be interpreted on~$\sr{n+1}$ in distributional sense.  Yet, from the practical point of view, an expression with classical meaning is far more preferable.
\end{remark}

It remains to verify that the inversion formula given by Equation~\eqref{e:ip} is well-defined which will then complete the proof of Theorem~\ref{t:ip}.  We do this in the next Section~\ref{s:radon} by deriving Equation~\eqref{e:ip} from the classical Radon inversion formula.

\section{Connection to Radon theory}
\label{s:radon}

In this section we are going to adhere to our earlier convention whereby we identify
\begin{math}
  \r{n+1} \equiv \r{n} \times \r{1}.
\end{math}
Accordingly, we shall parameterize points on the unit sphere
\begin{math}
  S^{n} \subset \r{n+1}
\end{math}
using pairs~$(\omega,\theta)$ with
\begin{math}
  \omega \in \r{n}
\end{math}
and
\begin{math}
  \theta \in \r{1}.
\end{math}
This may, unfortunately, make some of the equations listed below look unfamiliar.  Nevertheless, these are all straightforward adaptations of the formulas in classical Radon theory compiled from the first chapter of Helgason's text~\cite{Helgason-1999}.
\par
Let
\begin{math}
  \phi \in \sr{n+1}
\end{math}
be a rapidly decreasing function.  We introduce the classical $(n+1)$-dimensional Radon operator~$\radon$ by
\begin{equation} \label{e:radon}
  \radon(\phi)((\omega,\theta),p) =
  \int_{(\omega,\theta) \cdot (x,y) = p}
    \phi(x,y) \,
  dm(x,y).
\end{equation}
Notice that according to Equation~\eqref{e:radon} the Radon operator~$\radon$ takes
\begin{math}
  \phi \in \sr{n+1}
\end{math}
into its integrals over the $n$-planes in~$\r{n+1}$ parameterized by
\begin{math}
  (\omega,\theta) \cdot (x,y) = p.
\end{math}
In the parametrization of the planes the normal vector
\begin{math}
  (\omega,\theta) \in S^{n}
\end{math}
is the point on the unit sphere, split in accordance with our earlier explanation, while
\begin{math}
  p \in \r{1}
\end{math}
is the signed distance to the origin.  The measure~$dm(x,y)$ in Equation~\eqref{e:radon} is the $n$-dimensional Euclidean volume element in the plane.
\par
Now let
\begin{math}
  \psi = \radon(\phi).
\end{math}
Notice that, according to Equation~\eqref{e:radon}, $\psi$ is a function on
\begin{math}
  S^{n} \times \r{1}.
\end{math}
Following Helgason in~\cite{Helgason-1999}, we define the adjoint Radon transform~$\radon^{*}$ by
\begin{equation} \label{e:radon*}
  \radon^{*}(\psi)(x,y) =
  \int_{(\omega,\theta) \in S^{n}}
    \psi((\omega,\theta), (\omega,\theta) \cdot (x,y)) \,
  d(\omega,\theta),
\end{equation}
where~$d(\omega,\theta)$ is the normalized Euclidean volume element on~$S^{n}$.
\par
In Section~\ref{s:preliminaries} we introduced the Hilbert transform~$\hilbert_{y}$ and $\Lambda_{y}$-operator acting on the $y$-variable (Equations~\eqref{e:hy} and~\eqref{e:lambda}, respectively). In this section we will use the same operators but acting on the $p$-variable instead:
\begin{eqnarray*}
  \hilbert_{p}(\phi)(\cdot,p) &=&
    \frac{1}{\pi} \, \pv \,
    \int\limits_{-\infty}^{+\infty}
      \frac{\phi(\cdot,t)}{p - t} \\
  \Lambda_{p} &=&
    \begin{cases}
      (-1)^{\frac{n-1}{2}} \, (\hilbert_{p} \circ \partial_{p}^{n}), &n=1,3,5,\ldots, \\
      (-1)^{\frac{n}{2}} \, \partial_{p}^{n} &n=2,4,6,\ldots
    \end{cases}
\end{eqnarray*}
As can be readily deduced from Chapter~I of~\cite{Helgason-1999},\footnote{Helgason's definition of $\Lambda$-operator does not have the powers of $(-1)$ while his definition of Hilbert transform, unlike ours, has an imaginary unit as a multiple.} the following commutation relation holds:
\begin{equation} \label{e:lambda_comm}
  (\radon \circ \Lambda_{y})(\phi)((\omega,\theta),p) =
  \theta^{n} \, (\Lambda_{p} \circ \radon)(\phi)((\omega,\theta),p).
\end{equation}
Most importantly, $\Lambda$-operator occurs in the following variant of the classical Radon inversion formula
\begin{equation} \label{e:iradon} 
     \phi = c_{n} \, (\radon^{*} \circ \Lambda_{p} \circ \radon) (\phi), \quad
     c_{n} = (-2 \, \pi)^{-n} \, \frac{\pi^{\frac{n+1}{2}}}{\Gamma \left( \frac{n+1}{2} \right)},
\end{equation}
which strongly resembles the inversion formula for the convolution operator~$\norton$ stated in Theorem~\ref{t:ip} of Section~\ref{s:results}.
\par
We complete our review of classical Radon theory by taking note of the well-known \emph{translation property} of the operator~$\radon$.  Let~$\tau$ denote a shift of a function as in
\begin{math}
  \tau_{(a,b)}(\phi)(x,y) = \phi(x + a,y + b).
\end{math}
Then by making a straightforward substitution in Equation~\eqref{e:radon}, it is easy to show (see Page~$3$ in~\cite{Helgason-1999}) that
\begin{equation} \label{e:tau}
    ( \radon \circ \tau_{(a,b)} )(\phi)((\omega,\theta),p) =
      \radon(f)((\omega,\theta), p + (\omega,\theta) \cdot (a,b)).
\end{equation}
Now, in order to deduce Theorem~\ref{t:ip} from Equation~\eqref{e:iradon}, which is the principle goal of this section, we are going to require the following technical lemma.
\begin{lemma} \label{l:pullback}
Let
\begin{math}
  S_{-}^{n} = \{ (\omega,\theta) \in S^{n} \mid \theta < 0 \}
\end{math}
be the lower half of the unit sphere
\begin{math}
  S^{n} \in \r{n+1}.
\end{math}
Then
\begin{equation} \label{e:equal}
  \int_{\r{n}}
    h \left( \frac{(z,-1)}{\sqrt{1 + |z|^{2}}} \right) \,
  \frac{dz}{(1 + |z|^{2})^{\frac{n+1}{2}}} =
  |S^{n}| \,
  \int_{S_{-}^{n}}
    h(\omega,\theta) \,
  d(\omega,\theta),
\end{equation}
where~$d(\omega,\theta)$ is the normalized spherical measure (restricted to~$S_{-}^{n}$) while the constant~$|S^{n}|$ is the total $n$-dimensional volume of the unit sphere given by Equation~\eqref{e:sn}.
\end{lemma}
The proof of Lemma~\ref{l:pullback} is a straightforward, albeit involved computation; it is presented in Appendix~\ref{s:appendix1}. Meanwhile, we proceed to the proof of the main results of this section.
\par
Let
\begin{math}
  F \in \sr{n}
\end{math}
be a rapidly decreasing function and set
\begin{math}
  G = \norton(F)
\end{math}
as in Equation~\eqref{e:p}.  Theorem~\ref{t:validation} below validates the inversion formulas stated in Theorem~\ref{t:ip} of Section~\ref{s:results}.
\begin{theorem} \label{t:validation}
The inversion formula
\begin{displaymath}
  F = \pi^{-n} \, (\norton^{*} \circ \Lambda_{y})(G)
    = \pi^{-n} \, (\norton^{*} \circ \Lambda_{y} \circ \norton)(F),
\end{displaymath}
stated in Theorem~\ref{t:ip} in Section~\ref{s:results} is valid for rapidly decreasing functions
\begin{math}
  F \in \sr{n}.
\end{math}
\end{theorem}
\begin{proof}
Since~$F$ is rapidly decreasing, the operators~$\Lambda_{y}$ and~$\norton$ commute:
\begin{displaymath}
  (\Lambda_{y} \circ \norton)(F)(x,y) =
  (\norton \circ \Lambda_{y})(F)(x,y) =
    \int_{\r{n}}
      \Lambda_{y}(F)(x + u, y - |u|^{2}) \,
    du.
\end{displaymath}
Hence, for a fixed
\begin{math}
  (x_{0},y_{0}) \in \r{n} \times \r{1},
\end{math}
we can rewrite the right-hand side of the inversion formula as
\begin{displaymath}
  \pi^{-n} \,
  \int_{\r{n}}
  \left[
    \int_{\r{n}}
      \Lambda_{y}(F)(x_{0} + u + v, y_{0} + |v|^{2} - |u|^{2}) \,
    du
  \right] \,
  dv,
\end{displaymath}
which evidently is an integral in
\begin{math}
  \r{n} \times \r{n}.
\end{math}
Now the change of variables
\begin{displaymath}
    v + u = w, \quad
    v - u = z, \quad
    (|v|^{2} - |u|^{2} = (v + u) \cdot (v - u) = w \cdot z)
\end{displaymath}
transforms that integral into the expression
\begin{displaymath}
  (2 \, \pi)^{-n} \,
  \int_{\r{n}}
  \left[
    \int_{\r{n}}
      \Lambda_{y}(F)(x_{0} + w, y_{0} + z \cdot w) \,
    dw
  \right] \,
  dz,
\end{displaymath}
which we proceed to restate in terms of Radon transform~\eqref{e:radon}.  To see how the operator~$\radon$ enters the picture, rewrite the bracketed integral as
\begin{displaymath}
  \frac{1}{\sqrt{1 + |z|^{2}}} \,
  \int_{\r{n}}
    ( \tau_{(x_{0},y_{0})} \circ \Lambda_{y})(F)(w,z \cdot w) \,
  \sqrt{1 + |z|^{2}} \, dw.
\end{displaymath}
This shows that~$\Lambda_{y}(F)(x,y)$ is first shifted by~$(x_{0},y_{0})$, then integrated over the plane
\begin{math}
  (w, z \cdot w, w \in \r{n})
\end{math}
with respect to the natural volume element
\begin{math}
  \sqrt{1 + |z|^{2}} \, dw,
\end{math}
and, finally, scaled by
\begin{math}
  (1 + |z|^{2})^{-1/2}:
\end{math}
in sum,
\begin{displaymath}
  \frac{1}{\sqrt{1 + |z|^{2}}} \,
  (\radon \circ \tau_{(x_{0},y_{0})} \circ \Lambda_{y})(F)
  \left( \pm \frac{(z,-1)}{\sqrt{1 + |z|^{2}}}, 0 \right),
\end{displaymath}
where we have complete freedom in choosing the sign of the normal---an important symmetry to be exploited later.

Choosing the positive sign for now and using the translation property~\eqref{e:tau} combined with the $\Lambda$-commutation property~\eqref{e:lambda_comm}, we are led to
\begin{displaymath}
  \frac{(-1)^{n}}{(1+|z|^{2})^{\frac{n+1}{2}}} \,
  (\Lambda_{p} \circ \radon)(F)
  \left(
    \frac{(z,-1)}{\sqrt{1 + |z|^{2}}},
    \frac{(z,-1)}{\sqrt{1 + |z|^{2}}} \cdot (x_{0},y_{0})
  \right),
\end{displaymath}
which reduces our task to proving
\begin{displaymath}
  F(x_{0},y_{0}) =
  (-2 \, \pi)^{-n}
  \int_{\r{n}}
    \frac{
    (\Lambda_{p} \circ \radon)(F)
    \left(
      \frac{(z,-1)}{\sqrt{1 + |z|^{2}}}, \frac{(z,-1)}{\sqrt{1 + |z|^{2}}} \cdot (x_{0},y_{0})
    \right)}{(1+|z|^{2})^{\frac{n+1}{2}}} \, dz.
\end{displaymath}
We claim that the latter is equivalent to the Radon inversion formula~\eqref{e:iradon}.  To prove the claim and thus settle the theorem, we apply the mapping~$\Phi$ from Lemma~\ref{l:pullback} to change integration over~$\r{n}$ into integration over the lower half of the unit sphere~$S_{-}^{n}:$
\begin{displaymath}
  F(x_{0},y_{0}) =
  \frac{|S^{n}|}{(-2 \, \pi)^{n}} \,
  \int_{S_{-}^{n}} \,
    (\Lambda_{p} \circ \radon)(F)
    ((\omega,\theta), (\omega,\theta) \cdot (x_{0},y_{0}) ) \,
  d(\omega,\theta).
\end{displaymath}
We now observe that the integration over the \emph{lower} half-sphere~$S_{-}^{n}$ is the consequence of our earlier choice of the positive sign in front of the normal
\begin{displaymath}
  \pm \frac{(z,-1)}{\sqrt{1 + |z|^{2}}}
\end{displaymath}
of the plane
\begin{math}
  \{ w, z \cdot w, w \in \r{n} \}
\end{math}
whereas had we chosen the negative sign, we would have obtained instead the same integral over the \emph{upper} half-sphere~$S_{+}^{n}$ as immediately seen from Lemma~\ref{l:pullback}.  Therefore we can extend the domain of integration to the entire unit sphere~$S^{n}$:
\begin{displaymath}
  \begin{split}
    F(x_{0},y_{0})
  &=
    \frac{1}{2} \,
    \frac{|S^{n}|}{(-2 \, \pi)^{n}} \,
    \int_{S^{n}} \,
      (\Lambda_{p} \circ \radon)(F)
      ((\omega,\theta), (\omega,\theta) \cdot (x_{0},y_{0}) ) \,
    d(\omega,\theta)\\
  &=
    \frac{1}{2} \,
    \frac{|S^{n}|}{(-2 \, \pi)^{n}} \,
    (\radon^{*} \circ \Lambda_{p} \circ \radon)(F)(x_{0},y_{0}),
  \end{split}
\end{displaymath}
Now substitution of the area of the unit sphere~$S^{n}$ with its expression~\eqref{e:sn} from Section~\ref{s:preliminaries} leads to Radon inversion formula~\eqref{e:iradon}, as required.
\end{proof}

The intimate connection between~$\norton$ and~$\radon$ elucidated by Theorem~\ref{t:validation} suggests that instead of using Fourier transform to invert~$\norton$, as we did in Section~\ref{s:fourier}, we could instead use Radon transform.  Moreover, the use of Radon transform is, in some sense, more natural.\footnote{See also the discussion of applications of Radon transform to partial differential equations at the end of Chapter~I in~\cite{Helgason-1999}.} Indeed, Lemma~\ref{l:compact} in Section~\ref{s:range} shows that, unlike Fourier transform, the Radon transform in~$\r{n+1}$ can be freely applied to
\begin{math}
  G = \norton(F)
\end{math}
without recourse to distribution theory.  This motivates our derivation of a Radon-based inversion formula~\eqref{e:ip_radon} in Theorem~\ref{t:ip_radon} below.  First, however, we need to introduce some additional background.
\par
Let
\begin{math}
  \phi \in \mathscr{S} \left( S^{n} \times \r{1} \right):
\end{math}
think of $\phi$ as the Radon transform of a rapidly decreasing function in~$\r{n+1}$.  We now define a one-dimensional operator~$\q$ acting on the $p$-variable by
\begin{equation} \label{e:q}
  \q(\phi)((\omega,\theta),p) =
  \int_{\r{n}}
    \phi((\omega,\theta), p + (\omega,\theta) \cdot (u,-|u|^{2})) \,
  du.
\end{equation}
The adjoint operator~$\q^{*}$ is then given by
\begin{equation} \label{e:q*}
  \q^{*}(\phi)((\omega,\theta),p) =
  \int_{\r{n}}
    \phi((\omega,\theta), p + (\omega,\theta) \cdot (v,|v|^{2})) \,
  dv,
\end{equation}
while the inverse~$\q^{-1}$ is presented in Lemma~\ref{l:iq} below.

\begin{lemma} \label{l:iq}
Let~$\q$ be the operator defined by Equation~\eqref{e:q}.  Then
\begin{equation} \label{e:iq}
  \phi = \frac{|\theta|^{n}}{\pi^{n}} \, (\q^{*} \circ \Lambda_{p})(\phi), \quad
  \phi \in \mathscr{S} \left( S^{n} \times \r{1} \right),
\end{equation}
where~$\q^{*}$ is the adjoint operator defined by Equation~\eqref{e:q*}.
\end{lemma}
\begin{proof}
Applying one-dimensional Fourier transform in the $p$-variable
\begin{displaymath}
  \phi((\omega,\theta),p) \mapsto \widehat{\phi}(\xi) =
  \int\limits_{-\infty}^{\infty}
    \exp (-i \, \xi \, p) \, \phi((\omega,\theta), p) \,
  dp,
\end{displaymath}
to the convolution relation
\begin{math}
  \psi = \q(\phi),
\end{math}
we get the product relation
\begin{math}
  \widehat{\psi} = \widehat{\q} \, \widehat{\phi}
\end{math}
where the Fourier multiplier~$\widehat{\q}$ can be regularized with an exponential convergence factor as in Lemma~\ref{l:phat} in Section~\ref{s:fourier}:
\begin{displaymath}
  \begin{split}
    \widehat{\q} &=
      \int_{\r{n}}
        \exp ( i \, \xi \, (\omega,\theta) \cdot (u,-|u|^{2}) ) \,
      du\\ &=
      \lim_{\epsilon \to 0^{+}}
      \int_{\r{n}}
        \exp( i \, \xi \, \omega \cdot u - (\epsilon + i \, \xi \, \theta \, |u|^{2} ) ) \,
      du.
  \end{split}
\end{displaymath}
Now, as follows from Equation~\eqref{e:phat},
\begin{displaymath}
  \widehat{\q} =
    \left( \frac{\pi}{|\xi \, \theta|} \right)^{\frac{n}{2}} \,
    \exp \left( i \, \frac{|\xi \, \omega|^{2} - \pi \, n \, |\xi \, \theta|}{4 \, \xi \, \theta} \right)
\end{displaymath}
and the reciprocal relation
\begin{math}
  \widehat{\phi} = (1 / \widehat{\q}) \, \widehat{\psi}
\end{math}
can be written as
\begin{displaymath}
  \widehat{\phi} = \frac{|\theta|^{n}}{\pi^{n}} \,
  \left[
    \left( \frac{\pi}{|\xi \, \theta|} \right)^{\frac{n}{2}} \,
    \exp \left( -i \, \frac{|\xi \, \omega|^{2} - \pi \, n \, |\xi \, \theta|}{4 \, \xi \, \theta} \right)
  \right]
  \, |\xi|^{n} \, \widehat{\psi}.
\end{displaymath}
Recognizing the bracketed expression as
\begin{math}
  \overline{\widehat{\q}} = \widehat{\q^{*}}
\end{math}
and~$|\xi|^{n}$ as the familiar Fourier symbol of the $\Lambda$-operator (acting on the variable $p$), we are led to the statement of the lemma.
\end{proof}
Using Lemma~\ref{l:iq} and Radon inversion formula~\eqref{e:iradon}, it is easy to invert~$\norton$ as Theorem~\ref{t:ip_radon} demonstrates.
\begin{theorem} \label{t:ip_radon}
Let
\begin{math}
  F \in \sr{n+1}.
\end{math}
The following inversion formula holds
\begin{equation} \label{e:ip_radon}
  F = (\radon^{-1} \circ \q^{-1} \circ \radon \circ \norton)(F),
\end{equation}
where~$\q^{-1}$ is given by Equation~\eqref{e:iq} in Lemma~\ref{l:iq}.
\end{theorem}
\begin{proof}
Applying Radon transform~\eqref{e:radon} to Equation~\eqref{e:p}, interchanging the order of integrations, and using the translational property~\eqref{e:tau} of Radon transform~$\radon$, we obtain  the relation
\begin{displaymath}
  \radon(G)((\omega,\theta),p) =
  \int_{\r{n}}
    \radon(F)((\omega,\theta), p + (\omega,\theta) \cdot (u,-|u|^{2})) \,
  du,
\end{displaymath}
which we can rewrite in operator form as
\begin{displaymath}
  \radon(G) = (\q \circ \radon)(F).
\end{displaymath}
Now writing
\begin{math}
  G = \norton(F)
\end{math}
and applying the inverses~$\q^{-1}$ and~$\radon^{-1}$, in that order, we get the statement of the theorem.
\end{proof}

\section{Conclusions}
\label{s:conclusions}

The inversion formulas for the spherical mean transform~$\sonar$, presented in Sections~\ref{s:results} and~\ref{s:radon}, strengthen the existing results and shed new light on the relation between spherical means and classical Radon theory.  The Fourier-Radon techniques used to obtain our results have the virtue of simplicity, although this, admittedly, is a relative term.  Unfortunately, the same techniques are unlikely to apply to spherical mean transforms with curved centersets.  Indeed, from the group-theoretic point of view, our inversion of the operator~$\sonar$ hinges on replacing its symmetry group with a much simpler symmetry group of the operator~$\norton$: notice that the symmetries are being replaced, not created.  Yet, there are few centersets in~$\r{n}$ with sufficiently large symmetry groups. \footnote{In~$\r{2}$ all symmetric, or \emph{homogeneous}, centersets can be enumerated by classifying the one-parameter subgroups of the M\"obius group of the plane.  The complete list, presented in the author's doctoral thesis, is limited to lines, rays, circles, open intervals, and logarithmic spirals.}  In fact, flat centersets are the most symmetric, followed by spheres.  The symmetry group of a spherical centerset~$\operatorname{SO}(n)$ has lower dimension than that of~$\norton$.  This means that averages over the spheres centered on a sphere \( S^{n-1} \subset \r{n} \) cannot be straightforwardly related to \emph{any} convolution in~$\r{n}$.  Nevertheless, we consider our \emph{ad hoc} inversion worthy of attention.  The compositional structure of~$\sonar^{-1}$ exhibited in Equation~\eqref{e:is_norton} suggests a straightforward and effective FFT-based reconstruction algorithm.  We plan to pursue the development of such an algorithm and its numerical analysis elsewhere.
\par
It would also be of interest to extend the two Lemmas in Section~\ref{s:range} to a full range characterization of spherical means~$\sonar$.  We conjecture that a \emph{continuous} function~$G$ which has asymptotic behavior
\[
  G(x + v, y+|v|^{2}) = \bigo(|v|^{-1})
\]
given by Equation~\eqref{e:asymptotic} is the image of a rapidly decreasing function~$F$ under~$\norton$.  If, as in Lemma~\ref{l:compact}, the support of~$G$ can be covered by a union of parabolas with vertices ranging over a compact set, the pre-image~$F$ must be compactly supported.

\appendix
\section{Proof of Lemma~\ref{l:pullback}}
\label{s:appendix1}

\begin{proof}[Proof of Lemma~\ref{l:pullback}]
For notational convenience, let us define mapping
\begin{math}
  \Phi: \r{n} \mapsto S_{-}^{n}
\end{math}
by
\begin{equation*}
  \Phi(z) = \frac{(z,-1)}{\sqrt{1 + |z|^{2}}}.
\end{equation*}
Geometrically, $\Phi$ projects the points of the tangent space~$T_{(0,\ldots,0,-1)}(S^{n})$ onto~$S_{-}^{n}$ as illustrated by Figure~\ref{fig:Phi}; we may therefore call $\Phi$ a \emph{central projection}.

\begin{figure}[ht!]
  \begin{center}
    \scalebox{.85}{\includegraphics{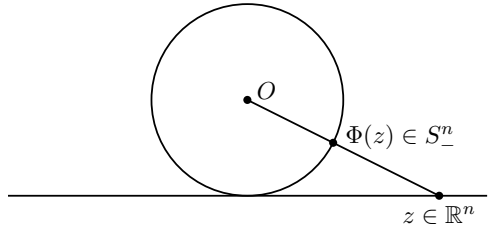}}
  \end{center}
  \caption{The central projection $\Phi$} \label{fig:Phi}
\end{figure}

Notice that~$\Phi$ is invertible with the inverse
\begin{math}
  \Phi^{-1}: S_{-}^{n} \mapsto \r{n}
\end{math}
given  by
\begin{displaymath}
  \Phi^{-1}((\omega,\theta)) = \frac{\omega}{|\theta|};
\end{displaymath}
notice further that~$\Phi$ is  smooth: it is therefore a diffeomorphism between the Euclidean space~$\r{n}$ and the unit half-sphere~$S_{-}^{n}$.
\par
We shall now introduce some additional terms from differential geometry.  Let~$\vol_{S^{n}}$ denote the Riemannian volume form on the unit sphere
\begin{math}
  S^{n} \subset \r{n+1}
\end{math}
(restricted to the lower half-sphere~$S_{-}^{n}$) and, likewise, let
\begin{math}
  \vol_{\r{n}}(z) = dz_{1} \wedge \ldots \wedge z_{n}
\end{math}
denote the natural volume in~$\r{n}$; we shall use the absolute value sign to distinguish a differential form from its corresponding density;  finally, we denote by
\begin{math}
  \Phi^{*}:  \Lambda^{k}(S_{-}^{n}) \mapsto \Lambda^{k}(\r{n})
\end{math}
the \emph{pull-back} mapping induced by
\begin{math}
  \Phi: \r{n} \mapsto S_{-}^{n}
\end{math}
on the space of $k$-forms.  We can now restate Equation~\eqref{e:equal} abstractly as
\begin{displaymath}
  \int_{\r{n}}
    (h \circ \Phi)(z) \,
  \left| \frac{\vol_{\r{n}}(z)}{(1 + |z|^{2})^{\frac{n+1}{2}}} \right| =
  \int_{\Phi(\r{n})}
    h((\omega,\theta)) \,
    \left| \vol_{S^{n}}(\omega,\theta) \right|,
\end{displaymath}
noting that the proof of the lemma reduces to showing the equality of two densities
\begin{displaymath}
  \left| \Phi^{*}(\vol_{S^{n}})(z) \right|  =
  \left| \frac{\vol_{\r{n}}(z)}{(1 + |z|^{2})^{\frac{n+1}{2}}} \right|,
\end{displaymath}
which will follow if~$\Phi^{*}(\vol_{S^{n}})(z)$ and~$(1 + |z|^{2})^{-(n+1)/2} \, \vol_{\r{n}}(z)$ are shown to agree up to sign.  Equivalently, we must show that
\begin{equation*}
  \frac{\Phi^{*}(\vol_{S^{n}})(z)}{\vol_{\r{n}}(z)} = \pm
    \frac{1}{(1 + |z|^{2})^{\frac{n+1}{2}}}.
\end{equation*}
Therefore we define~$\vol_{S^{n}}$ explicitly by
\begin{equation*}
  \begin{split}
    \vol_{S^{n}}(\omega,\theta)
  &=
    \sum_{i=1}^{n}
      (-1)^{n+1+i} \, \omega_{i} \,
      d\omega_{1} \wedge \ldots  \mathring{d\omega_{i}} \ldots \wedge d\omega_{n} \wedge d\theta\\
  &+
      \theta \, d\omega_{1} \wedge \ldots \wedge d\omega_{n},
  \end{split}
\end{equation*}
where the empty dot over~$d\omega_{i}$ indicates omission of the $i$-th term in the wedge product. \footnote{We use the ``empty dot'' accent rather than the more customary ``hat'' from differential geometry to stay true to our promise in Section~\ref{s:fourier} to reserve the ``hat''-symbol exclusively for the Fourier transform in~$\r{n}$}  Now from the definition of the pull-back of a mapping in differential geometry, we conclude that
\begin{displaymath}
  \begin{split}
    \frac{\Phi^{*}(\vol_{S^{n}})(z)}{\vol_{\r{n}}(z)}
  &=
    \sum_{i=1}^{n}
      (-1)^{n+1+i} \, \omega_{i}(z) \, \Phi^{*}
      \left(
        d\omega_{1} \wedge \ldots  \mathring{d\omega_{i}} \ldots \wedge d\omega_{n} \wedge d\theta
      \right)\\
  &+
      \theta(z) \, \Phi^{*}
      \left(
        d\omega_{1} \wedge \ldots \wedge d\omega_{n}
      \right) = \det(A),
  \end{split}
\end{displaymath}
where~$A$ is the following
\begin{math}
  (n+1) \times (n+1)
\end{math}
matrix
\begin{equation*}
  A =
  \left[
    \begin{array}{cccc}
      \frac{\partial \omega_{1}}{\partial z_{1}} & \hdots & \frac{\partial \omega_{1}}{\partial z_{n}} & \omega_{1} \\
      \vdots & \ddots & \vdots & \vdots \\
      \frac{\partial \omega_{n}}{\partial z_{1}} & \hdots & \frac{\partial \omega_{n}}{\partial z_{n}} & \omega_{n} \\
      \frac{\partial \theta}{\partial z_{1}} & \ldots & \frac{\partial \theta}{\partial z_{n}} & \theta \\
    \end{array}
  \right]
\end{equation*}
with the entries given explicitly by
\begin{equation*}
  \begin{split}
    \frac{\partial \omega_{i}}{\partial z_{j}}
  &=
      \frac{1}{(1 + |z|^{2})^{\frac{3}{2}}} \,
      \begin{cases}
        -z_{i} \, z_{j}, &i \neq j,\\
        \phantom{-}1 + |z|^{2} - z_{i}^{2}, &i = j.
      \end{cases}\\
    \frac{\partial \theta}{\partial z_{j}} &= \frac{z_{j}}{(1 + |z|^{2})^{\frac{3}{2}}}.
  \end{split}
\end{equation*}
as can be easily ascertained.  The proof of the lemma is now reduced to a determinant computation: namely, we must show that
\begin{displaymath}
  \det(A) = \pm \frac{1}{(1 + |z|^{2})^{\frac{n+1}{2}}}.
\end{displaymath}
We will find the determinant of~$A$ in two steps.  First, we introduce the matrix~$B$ by
\begin{displaymath}
  B =
  \left[
    \begin{array}{cccc}
      1 + |z|^{2} - z_{1}^{2} & \hdots & -z_{1} \, z_{n} & z_{1} \\
      \vdots & \ddots & \vdots & \vdots \\
      -z_{n} \, z_{1} & \hdots & 1 + |z|^{2} - z_{n}^{2} & z_{n} \\
      z_{1} & \hdots & z_{n} & -1 \\
    \end{array}
  \right]
\end{displaymath}
noticing that the first~$n$ columns of~$A$ are the corresponding columns of~$B$  scaled by
\begin{math}
  (1 + |z|^{2})^{-3/2}
\end{math}
while the last column in~$A$ is obtained from the last column of~$B$ via scaling by
\begin{math}
    (1 + |z|^{2})^{-1/2}.
\end{math}
Consequently, we have
\begin{displaymath}
  \det(A) = \frac{1}{(1 + |z|^{2})^{\frac{3\,n + 1}{2}}} \, \det(B),
\end{displaymath}
and our task therefore is to compute~$\det(B)$.  To this end, we introduce a third and final matrix
\begin{displaymath}
  C =
  \left[
    \begin{array}{cccc}
      z_{1}^{2} & \hdots & z_{1} \, z_{n} & -z_{1} \\
      \vdots & \ddots & \vdots & \vdots \\
      z_{n} \, z_{1} & \hdots & z_{n}^{2} & -z_{n} \\
      -z_{1} & \hdots & -z_{n} & 2 + |z|^{2} \\
    \end{array}
  \right]
\end{displaymath}
whose connection to~$B$ is expressed by the relation
\begin{displaymath}
  B = (1 + |z|^{2}) \, \id - C,
\end{displaymath}
where~$I$ is the
\begin{math}
  (n+1) \times (n+1)
\end{math}
identity matrix.  From the relation between~$B$ and~$C$ follows the relation between their determinants
\begin{displaymath}
  \det(B) = -P_{C}(1 + |z|^{2}),
\end{displaymath}
where~$P_{C}(1 + |z|^{2})$ is the value of the characteristic polynomial of~$C$ at~$(1 + |z|^{2})$.  Thus to find~$\det(B)$ we need the characteristic polynomial of~$C$ which prompts us to examine the eigenvalues of~$C$.
\par
It is straightforward to verify by direct computation that~$C$ has rank~$2$ with the following~$(n-1)$ eigenvectors sharing the zero eigenvalue:
\begin{displaymath}
  \left[
    \begin{array}{c}
      -z_{2} \\
      \phantom{-}z_{1} \\
      \phantom{-}0 \\
      \phantom{-}\vdots \\
      \phantom{-}0 \\
      \phantom{-}0 \\
    \end{array}
  \right], \quad
  \left[
    \begin{array}{c}
      -z_{3} \\
      \phantom{-}0\\
      \phantom{-}z_{1} \\
      \vdots \\
      \phantom{-}0\\
      \phantom{-}0\\
    \end{array}
  \right], \ldots, \quad
  \left[
    \begin{array}{c}
      -z_{n} \\
      \phantom{-}0\\
      \phantom{-}0\\
      \phantom{-}\vdots\\
      \phantom{-}z_{1}\\
      \phantom{-}0\\
    \end{array}
  \right] \in \r{n+1};
\end{displaymath}
while the two remaining eigenvectors of~$C$
\begin{displaymath}
  \left[
    \begin{array}{c}
      z_{1} \\
      z_{2}\\
      z_{3}\\
      \phantom{-}\vdots\\
      z_{n}\\
      -1 \pm \sqrt{1 + |z|^{2}} \\
    \end{array}
  \right] \in \r{n+1}
\end{displaymath}
have conjugate eigenvalues
\begin{math}
  1 + |z|^{2} \mp \sqrt{1 + |z|^{2}}.
\end{math}
We conclude that the characteristic polynomial of~$C$ is given explicitly by
\begin{displaymath}
  P_{C}(\lambda) = \det(C - \lambda \, \id) =
  \lambda^{n-1} \, \left( (\lambda - 1 - |z|^{2})^{2} - 1 - |z|^{2} \right).
\end{displaymath}
Hence,
\begin{displaymath}
  \det(B) = -p_{C}(1 + |z|^{2}) = (1 + |z|^{2})^{n},
\end{displaymath}
which implies
\begin{displaymath}
  \det(A) = \frac{1}{(1 + |z|^{2})^{\frac{3\,n + 1}{2}}} \, (1 + |z|^{2})^{n} =
  \frac{1}{(1 + |z|^{2})^{\frac{n + 1}{2}}},
\end{displaymath}
as desired.
\end{proof}

\bibliographystyle{amsplain}

\providecommand{\bysame}{\leavevmode\hbox to3em{\hrulefill}\thinspace}
\providecommand{\MR}{\relax\ifhmode\unskip\space\fi MR }
\providecommand{\MRhref}[2]{%
  \href{http://www.ams.org/mathscinet-getitem?mr=#1}{#2}
}
\providecommand{\href}[2]{#2}

\end{document}